\documentclass[12pt]{article}
\usepackage{amsmath,amssymb,amsbsy,amsfonts,amsthm,latexsym, amsopn,amstext,amsxtra,euscript,amscd,amsthm,url}

\usepackage[usenames, dvipsnames]{color}
\usepackage[utf8]{inputenc}
\usepackage[english]{babel}
\usepackage{mathrsfs}
\usepackage{caption}
\newtheorem{thm}{Theorem}
\newtheorem{rem}{Remark}

\newtheorem{lem}{Lemma}
\newtheorem{cor}{Corollary}

\renewcommand{\d}{\displaystyle}  
\def\l({\left(}  \def\r){\right)}
 
\DeclareMathOperator{\1}{\textbf{1}}

\DeclareMathOperator{\id}{id}
\DeclareMathOperator{\si}{si}
\global\long\def\epsilon{\varepsilon} 
\usepackage{amsmath}
\DeclareMathAlphabet{\mathpzc}{OT1}{pzc}{m}{it}

\begin{document}
\date{}

\title{\bf  The summatory function of the M\"{o}bius function involving the greatest common divisor}   
\author{Isao Kiuchi and Sumaia Saad Eddin}

\maketitle
{\def\thefootnote{}
\footnote{{\it Mathematics Subject Classification 2020: 11N37, 11A25, 11M99,    }}

\begin{abstract}
Let $\gcd(m,n)$ denote the greatest common divisor of the positive integers $m$ and $n$, and let $\mu$ represent the M\" obius function. For any real number $x>5$, we define the summatory function of the M\" obius function involving the greatest common divisor as
$
S_{\mu}(x) := \sum_{mn\leq x} \mu(\gcd(m,n)).   
$
In this paper, we present an asymptotic formula for $S_{\mu}(x)$. Assuming the Riemann Hypothesis, we delve further into the asymptotic behavior of $S_{\mu}(x)$ and derive a mean square estimate for its error term. Our proof employs the Perron formula, Parseval's theorem, complex integration techniques, and the properties of the Riemann zeta-function.
\end{abstract}

\maketitle

%%%%%%%%%%%%%%%%%%%%%%%%%%%%%%%%%%%%%%%%%%%%%%%%%%%%%%%%%%%%%%%%%%%%%%%%%%%%%%%%%%%%%%%%%%%%%%%%%%%%%%%%%%%%%%%
\section{Introduction and main results}

%%%%%%%%%%%%%%%%%%%%%%%%%%%%%%%%%%%%%%%%%%%%%%%%%%%%%%%%%%%%%%%%%%%%
%%%%%%%%%%%%%%%%%%%%%%%%%%%%%%%%%%%%%%%%%%%%%%%

Let $\mu(n)$ denote the M\"obius function, namely
\begin{equation*}
	\mu(n)=\left\{\begin{array}{cl} 
					 \d 1 &\ \ \ {\rm if} \ n=1, \ \ \mbox{} \\
					 \d (-1)^k &\ \ \ {\rm if} \ n\ {\rm is\ squarefree\ and}\ n=p_1p_2\cdots p_k, \ \ \mbox{} \\
					 \d 0 &\ \ \ {\rm if} \ n\ {\rm is\ not\ squarefree}. \\ 
				 \end{array} \right. 
\end{equation*}  
For any integer $k\geq 2$, let $\gcd(d_{1},\ldots,d_{k})$ denote the greatest common divisor of the positive integers $d_{1},\ldots,d_{k}$.  We define the Dirichlet convolution of two arithmetic functions $f$ and $g$ as  $f * g (n) = \sum_{d \mid n} f(d) g\left({n}/{d}\right)$  
for all positive integers $n$. 
The arithmetic functions $\1(n)$ and $\id (n)$ are defined by $\1(n) = 1$ 
and $\id (n)=n$, respectively. 
The function $\tau_k$ is the $k$-factors Piltz divisor function given by $\1*\1*\cdots \1$. Specifically, for $k=2$, we denote $\tau_2$ as $\tau$. For an arbitrary arithmetic function $g$, we define the sum
\begin{align*}                                                      
f_{g,k}(n):= \sum_{d_{1}\cdots d_{k}=n}g\left(\gcd(d_{1},\ldots, d_{k})\right)
\end{align*}
and its summatory function as
\begin{align} 
\label{Sums}
S_{g,k}(x):=\sum_{n\leq x} f_{g,k}(n)=M_{g,k}(x)+E_{g,k}(x),  
\end{align}  
where $M_{g, k}(x)$ and $E_{g,k}(x)$ denote the main term and the error term of $S_{g,k}(x)$, respectively. 
In 2012, Kr\"{a}tzel, Nowak and T\'{o}th \cite{KNT} provided asymptotic formulas for a class of arithmetic functions that describe the value distribution of the greatest common divisor function, generated by a Dirichlet series whose analytic behavior is determined by the factor $\zeta^2(s)\zeta(2s-1),$ where $\zeta(s)$ is the Riemann zeta-function. They proved that 
\begin{equation}
\label{S}
 f_{g,k}(n) = \sum_{a^{k}b=n}\mu*g(a)\tau_{k}(b).
\end{equation}
The identity \eqref{S} was instrumental in establishing asymptotic formulas for $f_{g,k}(n)$ for specific choices of $g$, such as the identity function ${\rm id}$ and the sum of divisors function $\sigma={\rm id}*\1$. 
Furthermore, they showed that the Dirichlet series generated by $ f_{g,k}(n)$ is expressed by    
\begin{align}                                                       
\label{D}                              
\sum_{n=1}^{\infty} \frac{f_{g,k}(n)}{n^s} 
&=  \frac{\zeta^k(s)}{\zeta(ks)}\sum_{n=1}^{\infty}\frac{g(n)}{n^{ks}}, 
\end{align}
which converges absolutely in the half-plane $\Re(s) >\sigma_{0}$,  where $\sigma_{0}$ depends on $g$ and $k$.

Now, we set $g=\mu$, and use the identity 
$
\sum_{n=1}^{\infty}\frac{\mu(n)}{n^s}=\frac{1}{\zeta^{}(s)}
$ 
to derive 
\begin{align}                                                        \label{mu-k}     
S_{\mu,k}(x)  &= \sum_{m\leq x^{\frac1k}}\mu*\mu(m)\sum_{n\leq \frac{x}{m^k}}\tau_{k}(n)  
\end{align}
and that  
\begin{align}                                                     
\sum_{n=1}^{\infty} \frac{f_{\mu,k}(n)}{n^s} &=  \frac{\zeta^k(s)}{\zeta^{2}(ks)},          
\label{mu-D}
\end{align}
where both series converge absolutely in the half-plane $\Re (s) > 1$.\\
Throughout the article when $k=2$, we adopt simplified notation: 
$f_{\mu, 2}(n)=f_\mu(n), S_{\mu, 2}(x)=S_{\mu}(x), M_{\mu, 2}(x)=M_{\mu}(x),$ and $E_{\mu, 2}(x)=E_{\mu}(x).$\\

The first purpose of this paper is to establish the following result which provides an asymptotic formula of $S_{\mu,2}(x).$
 \begin{thm}       
 \label{th1}     
For any real number $x>5$,  we have  
\begin{align*}                        
S_{\mu}(x)
= \frac{1}{\zeta^{2}(2)}\left(\log x + 2\gamma -1 - 4\frac{\zeta^{\prime}(2)}{\zeta(2)}\right)x+O\left(x^{1/2}\log^2 x \right),
\end{align*}
where $\gamma$ is the Euler-Mascheroni constant and $\zeta^{\prime}(s)$ is the first derivative of the zeta-function $\zeta(s).$  
\end{thm}
The proof is based on applying the Perron formula (which gives a direct link between the summatory
function $\sum_{n\leq x}a_n$ and the corresponding Dirichlet series $\alpha(x)=\sum_{n\geq 1}a_n
/n^s$, for more details see Section \ref{Section2}), contour integration to the generating Dirichlet series and using the properties of the Riemann zeta-function. \\
In the following result, we focus on estimating the error term and employ elementary methods,  including the use of a specific Lemma \ref{lem21} below (in Section \ref{Section2}), to achieve a better bound.
 \begin{thm}       
 \label{th31}     
For any real number $x>5$,  we have  
\begin{align}                        
\label{S-221}                               
E_{\mu}(x)=
O\l(x^{1/2}{\rm exp}\l(-A\l(\log x \log \log  x\r)^{1/3}\r) \r)  
\end{align}
with a positive constant $A$.
\end{thm}
Next, we assume that the Riemann Hypothesis is true. By using the complex integration approach and Montgomery--Vaughan's method \cite{MV1}, we obtain the following result.
 \begin{thm}      
 \label{th4}             
If the Riemann hypothesis is true, then,  
for any real number $x>5$, we have  
\begin{align*}                                                        
& E_{\mu}(x) =  \sum_{\ell\leq Y}\mu*\mu(\ell)\Delta\l(\frac{x}{\ell^2}\r) 
+ O\left(x^{\frac12+\varepsilon}Y^{\frac{1-4\beta}{4-4\beta}+\varepsilon}  \right)          
%\label{S-22rh} 
\end{align*}
with $x^{\varepsilon} \ll Y \ll x^{\frac12-\varepsilon}$, and a fixed number $\beta$ such that $\frac14 < \beta <\frac12$. Here 
$\Delta(x)$ denote the error term in the Dirichlet divisor problem.
\end{thm}
By setting $Y=x^{\frac{77}{258}}$ and $\beta=\frac25$ in the above equation and applying the given estimation of $\Delta(x)$ from \eqref{Delta}, we deduce the following result. 
\begin{cor}
\label{cor1}
We have
$$
E_{\mu}(x) \ll x^{\frac{439}{1032}+\varepsilon}.
$$
\end{cor}
Note that 
$
\frac{439}{1032}=\frac12-\frac{77}{1032}=0.4253875\cdots <\frac12=0.5.
$ 
This implies that the latest estimation of the error term yields a slight improvement in the order of magnitude of \eqref{S-221}. 

Now, we  define the integral of  the  mean square  of  $E_{\mu}(x)$  as follows
\begin{align*}                                                   
I_{\mu}(T) := \int_{1}^{T}E_{\mu}^{2}(t)dt  
\end{align*} 
for any real number $T>5$. 
Lastly, we use Parseval's identity (which establishes a connection between the sum or integral of squares of a function and the sum or integral of squares of its Fourier transform, see (A.5) in \cite{I}) along with certain properties of the Riemann zeta-function to derive the mean square estimates for $E_{\mu}(t)$. Further, we shall investigate the integral $I_{\mu}(T)$ assuming the Riemann Hypothesis.
 \begin{thm}      
 \label{th5}                   
If the Riemann Hypothesis is true, then,
for any real number $T>5$, we have  
\begin{align*}                         
I_{\mu}(T) \ll  T^{3/2}{\rm exp}\l(A\frac{\log T}{\log\log T}\r),               
\end{align*}
with a positive constant $A$.  
\end{thm}

\begin{rem} 
 Under the Riemann Hypothesis, if  $I_{\mu}(T)$ can be evaluated as
$$ 
 c_{\mu} T^{3/2}(\log T)^{\ell} + O\l(T^{3/2}(\log T)^{\ell-\delta}\r) 
$$ 
as $x\to \infty$ with a constant $c_{\mu}>0$ depending on $\mu$ 
and numbers $\ell>0$, $\delta>0$,  then we can derive the Omega-result for $E_{\mu}(x)$, namely 
\begin{align*}                               
E_{\mu}(x) = \Omega\l(x^{1/4}(\log x)^{\ell/2}\r).
\end{align*} 
\end{rem} 
%%%%%%%%%%%%%%%%%%%%%%%%%%%%%%%%%%%%%%%%%%%%%%%%%%%%%%%%%%%%%%%%%%%%%%%%%%%%%%%%%%%%%%%%%%%%%%%%%%%%%%%%%%%%%%%%%%%%%%%%%%%%%%
\subsection*{Sums of squarefree numbers } 
%%%%%%%%%%%%%%%%%%%%%%%%%%%%%%%%%%%%%%%%%%%%%%%%%%%%%%%%%%%%%%%%%%%%%%%%%%%%%%%%%%%%%%%%%%%%%%%%%%%%%%%%%%%%%%%%%%%%%%%%%%%%%
We set  $g=|\mu|$ in  \eqref{S}, yielding
\begin{align}                                                    
S_{|\mu|,k}(x)  &= \sum_{m\leq x^{1/k}}\mu*|\mu|(m)\sum_{n\leq {x}/{m^k}}\tau_{k}(n).   
\label{mu-k1}
\end{align}
Using \eqref{D} and the Dirichlet  series 
$
\sum_{n=1}^{\infty}\frac{|\mu(n)|}{n^s}=\frac{\zeta(s)}{\zeta^{}(2s)},
$ 
we deduce    
\begin{align}                                                     
\sum_{n=1}^{\infty} \frac{f_{|\mu|,k}(n)}{n^s} &=  \frac{\zeta^k(s)}{\zeta^{}(2ks)},          
\label{mu-D1}
\end{align}
which converges absolutely in the half-plane $\Re (s) > 1$. From \eqref{mu-D1}, we derive an alternative expression for \eqref{mu-k1}, as follows:
\begin{align}                                                                                
\label{mu-k2}
S_{|\mu|,k}(x)  &= \sum_{m\leq x^{{1}/{(2k)}}}\mu(m)\sum_{n\leq {x}/{m^{2k}}}\tau_{k}(n).  
\end{align} 
Taking $k=2$ in \eqref{mu-k2} and using \eqref{s-full3} and  \eqref{Delta} along with the Dirichlet series 
$
\sum_{n=1}^{\infty}\frac{\mu(n)}{n^s}=\frac{1}{\zeta(s)}
$
and 
$
\sum_{n=1}^{\infty}\frac{\mu(n)}{n^s}\log n =\frac{\zeta^{\prime}(s)}{\zeta^{2}(s)},
$
we obtain    
\begin{align*}                                                                   
S_{|\mu|,2}(x) 
&=  \sum_{m\leq x^{{1}/{4}}}\mu(m)\sum_{n\leq {x}/{m^{4}}}\tau_{}(n)  \nonumber \\
&=  x\l(\log x  +2\gamma -1 \r)\sum_{m\leq x^{1/4}}\frac{\mu(m)}{m^4}
-2x\sum_{m\leq x^{1/4}}\frac{\mu(m)}{m^4}\log m   +  \sum_{m\leq x^{1/4}}\mu(m)\Delta\l(\frac{x}{m^4}\r)  \nonumber \\
&=\frac{x}{\zeta(4)}\l(\log x  +2\gamma -1 -2\frac{\zeta^{\prime}(4)}{\zeta(4)}\r) + O\l(x^{\theta+\varepsilon}\r).
\end{align*}
Therefore, we have 
\begin{equation*}
%\label{mu-2a}  
  S_{|\mu|,2}(x) =\frac{x}{\zeta(4)}\left(\log x  +2\gamma -1 -2\frac{\zeta^{\prime}(4)}{\zeta(4)}\right) + O\left(x^{\theta+\varepsilon}\right)
\end{equation*}
with
$
\theta=\frac{131}{416} 
$
and any real number $\varepsilon >0$.
%%%%%%%%%%%%%%%%%%%%%%%%%%%%%%%%%%%%%%%%%%%%%%%%%%%%%%%%%%%%%%%%%%%%%%%%%%%%%%%%%%%%%%%%%%%%%%%%%%%%%%%%%%%%%%%%%%%%%%%%%%%%%%%%%%
\section{Some known results}
\label{Section2}
In this section, we recall the material we need to prove our results. Throughout this section and in subsequent ones, we adhere to the traditional notation $s=\sigma+it$.   
\begin{lem} 
\label{lem1}
For $t\geq t_{0} >0$ uniformly in $\sigma$, we have  
\begin{align*}
\zeta(\sigma+it)
&\ll \left\{\begin{array}{cl} 
\d  t^{(3-4\sigma)/6}\log t            &  \  \ \l(0\leq \sigma \leq 1/2\r),     \   \smallskip  \\
\d  t^{(1-\sigma)/3}\log t             &  \  \ \l(1/2  \leq  \sigma \leq 1\r),  \   \smallskip  \\
\d  \log t                                  &  \  \ \l(1\leq \sigma < 2\r).     \   \smallskip  
\end{array} \right.
\end{align*}
\end{lem}

\begin{proof}
The above formulas follow from  \cite[Theorem II.3.8]{Te}, and \cite[Theorem 8.5 and (8.112)]{I}.  
\end{proof}

\begin{lem}
\label{lem21}
For any  number $x>5$, we have 
\begin{align*}                       
%\label{ik1}
\sum_{n\leq x}\frac{\mu*\mu(n)}{n^2} = \frac{1}{\zeta^{2}(2)} + O\l(\frac{\delta(x)}{x}\r)
\end{align*}
and 
\begin{align*}                       
%\label{ik2}
\sum_{n\leq x}\frac{\mu*\mu(n)}{n^2}\log n = 2\frac{\zeta^{\prime}(2)}{\zeta^{3}(2)} + O\l(\frac{\delta(x)}{x}\r),
\end{align*}
 where 
$
\delta(x)= {\rm exp}\l(-c\l(\log x \log\log x\r)^{1/3}\r) 
$ 
with $c$ being a positive constant. 
\end{lem}
\begin{proof}
The proof of this result can be found in \cite[Lemma 3.3]{IK}. 
\end{proof}

%++++++++++++++++++++++++++++++++++++++++   Lemma 2.3    ++++++++++++++++++++++++++++++++++++++++++++++++++++++++++

The following two well-known results about the properties of the Riemann Zeta-function, which provides the functional equation of the Riemann Zeta-function and its bounds under the Riemann Hypothesis. The proof of these results can be found in \cite{I}, \cite{MV}, and \cite{T}.
\begin{lem} 
\label{lem2}
The Riemann zeta-function  $\zeta(s)$ can be analytically continued to  a meromorphic function in the whole complex plane ${\mathbb C}$, 
its only singularity being a  simple pole at $s=1$ with residue $1$. 
It satisfies a functional equation 
$$\zeta(s)=\chi(s)\zeta(1-s),$$
where $\chi(s)=2^{s}\pi^{s-1}\sin\frac{\pi s}{2}\Gamma(1-s).$ Also, in any bounded vertical strip, using the Stirling formula, we deduce 
$$\chi(s) =\left(\frac{t}{2\pi}\right)^{\frac12-\sigma-it}e^{i(t+\frac{\pi}{4})}\left(1+O\left(\frac{1}{t}\right)\right)  \qquad   (t\geq t_{0}>0).$$ 
\end{lem}
%%%%%%%%%%%%%%%%%%%%%%%%%%%%%%%%%%%%%%%%%%%%%%%%%%%%%%%%%%%%%%%%%%%%%%%%%%%%%%%%%%%%%%%%%%%%%%
\begin{lem}   \label{lem3}  
Assume that the Riemann Hypothesis is true.                                                        
For $t\geq t_{0}\geq 0$, uniformly in $\sigma$,  we have   
\begin{align*}                                                                                                   
\zeta(\sigma+it) 
&\ll  \left\{\begin{array}{cl} 
t^{\varepsilon}                                     &  \ \    \left(1/2 \leq  \sigma \leq 2 \right),     \\ 
1                                                   &  \ \    \left(\sigma > 2 \right),     
\end{array} \right.
\end{align*}
and 
\begin{align*}                                                                                                   
\frac{1}{\zeta(\sigma+it)} 
&\ll  \left\{\begin{array}{cl} 
t^{\varepsilon}                                     &  \ \    \left(\frac12 +\frac{1}{\log\log t} \leq  \sigma \leq 2 \right),      \\ 
1                                                   &  \ \    \left(\sigma > 2 \right).    
\end{array} \right.
\end{align*} 
\end{lem}
%%%%%%%%%%%%%%%%%%%%%%%%%%%%%%%%%%%%%%%%%%%%%%%%%%%%%%%%%%%%%%%%%%%%%%%%%%%%%%%%%%%%%%%%%%%%%%%%%%%%%%%%%%%%%%%%%%%%%%%%%%%%%%%%%
\subsection*{Perron's formula}
The Perron formula is used to express a function defined by a Dirichlet series $\alpha(s)=\sum_{n=1}^{\infty}a_n/n^s$, with abscissa of convergence $\sigma_c$, in terms of its complex integral. That means 
$$\sum_{n\leq x}a_n=\frac{1}{2\pi i}\int_{\sigma_0-i\infty}^{\sigma_0+i\infty}\alpha(s)\frac{x^s}{s}\, ds,$$
for $\sigma_0>\max(\sigma_c,0)$. 
\begin{thm}[Perron's formula, Theorem 5.1 in \cite{MV}]
\label{Perron1}
If $\sigma_0>\max(\sigma_c,0)$ and $x>0$, then 
$${\sum_{n\leq x}}^{\prime}a_n=\lim_{T\rightarrow \infty}\frac{1}{2\pi i}\int_{\sigma_0-iT}^{\sigma_0+iT}\alpha(s)\frac{x^s}{s}\, ds,$$
where $\sum^{\prime}$ indicates that if $x$ is an integer, then the last term into the sum becomes $a_x/2$ instead of $a_x$.
\end{thm}
\begin{thm}[Theorem 5.2 in \cite{MV}]
\label{Perron2}
If $\sigma_0>\max(\sigma_c,0)$ and $x>0$, then 
$${\sum_{n\leq x}}^{\prime}a_n=\frac{1}{2\pi i}\int_{\sigma_0-iT}^{\sigma_0+iT}\alpha(s)\frac{x^s}{s}\, ds+R,$$
where 
$$
R=\frac{1}{\pi}\sum_{x/2<n<x}a_n \si \left(T\log \frac{x}{n} \right)-\frac{1}{\pi}\sum_{x<n<2x}a_n \si \left(T\log \frac{x}{n} \right)+O\left( \frac{4^{\sigma_0}+x^{\sigma_0}}{T}\sum_n\frac{|a_n|}{n^{\sigma_0}}\right).$$
Here $\si (x)$ is the  sine integral and equal to $-\int_{x}^{\infty}\frac{\sin u}{u}.$ 
\end{thm}
\begin{cor}[Corollary 5.3 in \cite{MV}]
\label{Perron3}
Under the conditions stated above, we have
$$R\ll \sum_{\substack{x<n<2x, \\n\neq x}}|a_n| \min\left(1, \frac{x}{T|x-n|} \right)+\frac{4^{\sigma_0}+x^{\sigma_0}}{T}\sum_{n=1}^{\infty}\frac{|a_n|}{n^{\sigma_0}}.$$
\end{cor}
%%%%%%%%%%%%%%%%%%%%%%%%%%%%%%%%%%%%%%%%%%%%%%%%%%%%%%%%%%%%%%%%%%%%%%%%%%%%%%%%%%%%%%%%%%%%%%%%%%%%%%%%%%%%%%%%%%%%%%%%%%%%%
\section{Proofs} 

\subsection{Proof of Theorem \ref{th1}} 

Without loss of generality we may suppose that $x \in \mathbb{Z}+\frac12$. 
Applying Theorem \ref{Perron2} and Corollary \ref{Perron3} and using \eqref{mu-D} with $k=2$, we obtain 
\begin{align*}                                                                                    
%\label{AA}
S_{\mu,2}(x) &= \frac{1}{2\pi i} \int_{\alpha-iT}^{\alpha+iT} \frac{\zeta^{2}(s)}{\zeta^{2}(2s)} \frac{x^s}{s^{}}ds 
+ O\left(\frac{x^{1+\varepsilon}}{T^{}}\right),  
\end{align*}
where $\alpha= 1 +\frac{1}{\log x}$ and $T$ is a real parameter at our disposal.   \\
Let $\Gamma(\alpha, \beta,T)$ denote the contour consisting of the line 
segments $[\alpha-iT, \beta-iT]$, $[\beta-iT, \beta+iT]$ and $[\beta+iT, \alpha+iT]$. 
We denote the integrals over the horizontal line segments as $K_{1}$ and $K_{3}$, and the integral over the vertical line segment as $K_{2}$. Setting $\beta=\frac{1}{2}$, and using Lemma \ref{lem1} along with the estimate
$
\frac{1}{|\zeta(\sigma+it)|} \ll  \log (1+ |t|)
$
for $|t|\geq t_{0}>0$ and $\sigma\geq 1$,  we find
\begin{align*}                                                                                          
%\label{Q1}            
K_{1}, K_{3} 
&\ll  \frac{\log^{2}T}{T} \int_{\frac12}^{\alpha}\nonumber
|\zeta(\sigma+iT)|^{2}  x^{\sigma}d\sigma    \\
&\ll \frac{\log^{4}T}{T^{1/3}} \int_{\frac12}^{\alpha} \left(\frac{x}{T^{2/3}}\right)^{\sigma} d\sigma    \nonumber   \\
&\ll \log^{4}T \left(\frac{x}{T}  + \frac{x^{1/2}}{T^{2/3}}\right).    
\end{align*}
For the integral $K_{2}$, using the estimate given in \cite{Y} that 
$$\int_{1}^{T}\frac{|\zeta(\frac12+it)|^{2}}{|\zeta(1+2it)|^{2}}\, dt \ll T\log T$$ 
and applying integration by parts, we obtain
$$
\int_{1}^{T}\frac{|\zeta(\frac12+it)|^{2}}{|\zeta(1+2it)|^{2}} \frac{ dt}{t}
\ll  T\log T + \int_{T/2}^{T}\frac{\log t}{t}dt \ll \log^{2} T.
$$
Thus, employing Lemma \ref{lem1}, we have
\begin{align*}                                                                                        
%\label{ki1}
K_{2} 
&\ll x^{1/2} + x^{1/2}
 \int_{t_{0}\leq |t|\leq T} \frac{|\zeta(\frac12+it)|^{2}}{|\zeta(1+2it)|^{2}} \frac{dt}{t}  \\
&\ll x^{1/2} \log^{2}T.    \nonumber
\end{align*}
To complete the proof, we calculate the residue at the double pole $s=1$ with the residue 
$$
\underset{s=1}{\rm{Res}}~\frac{\zeta^{2}(s)}{\zeta^{2}(2s)} \frac{x^{s}}{s},  
$$
which equals
$$
\frac{1}{\zeta^{2}(2)}\l(\log x + 2\gamma -1 - 4\frac{\zeta'(2)}{\zeta(2)}\r)x. 
$$
Therefore, we conclude that
\begin{align*}                                                                                      
S_{\mu,2}(x) = \frac{1}{\zeta^{2}(2)}\l(\log x + 2\gamma -1 - 4\frac{\zeta'(2)}{\zeta(2)}\r)x +   O\left( x^{1/2}\log^{2}x \right),  
\end{align*}
and by setting $T=x^{}$, the proof is complete.

%%%%%%%%%%%%%%%%%%%%%%%%%%%%%%%%%%%%%%%%%%%%%%%%%%%%%%%%%%%%%%%%%%%%%%%%%%%%%%%%%%%%%%%%%%%%%%%%%%%%%%%%%%%%%%%%%%%%%%%%%%%%%%%%%%%%%%%%%%%%%%%%%%%%%%%%%%%%%%%%%%%%%%%%%%%%%%%%%%%%%%%%%%%%%%%%%%%%%%%%%%%%%%%%%%%%%%%%%%%%%%%%%

\subsection{Proof of Theorem~\ref{th31}} 

%+++++++++++++++++++++++++++++++++++++++++++++++++++++++++++++++++++++++++++++++++++++++++++++++++++++++++++
Let $\tau(n)$ denote the  divisor function, defined by ${\bf 1}*{\bf 1}$,  
and let $\Delta(x)$ represent the error term in the asymptotic formula for the sum
\begin{align}                                                                                           
\label{s-full3}
&D(x) := \sum_{n\leq x}\tau(n)  = \l(\log x + 2\gamma -1\r)x + \Delta(x)  
\end{align}
valid for any real number $x>2$. It is known the Dirichlet divisor problem involves studying the properties of $\Delta(x)$. The best known estimate 
\begin{align}                                                                                           \label{Delta}
\Delta(x) \ll x^{\theta+\varepsilon},
\end{align} 
with $\theta=\frac{131}{416}$, was proved by Huxley \cite{H}. 
Using \eqref{s-full3} and \eqref{mu-k} with $k=2$, we obtain
\begin{align*}
S_{\mu}(x)&=\sum_{m\leq x^{1/2}}\mu*\mu(m)\sum_{n\leq {x}/{m^2}}\tau(n) \\
&=x\l(\log x  +2\gamma -1 \r)\sum_{m\leq x^{1/2}}\frac{\mu*\mu(m)}{m^2}
-2x\sum_{m\leq x^{1/2}}\frac{\mu*\mu(m)}{m^2}\log m 
+  U(x),  
\end{align*}
where 
$$
U(x):=\sum_{m\leq x^{1/2}}\mu*\mu(m)\Delta\l(\frac{x}{m^2}\r).
$$
Applying Lemma \ref{lem21}, we have 
\begin{align*}                       
%\label{SSS}
S_{\mu}(x) 
&=\frac{x}{\zeta^{2}(2)}\l(\log x  +2\gamma -1 -4 \frac{\zeta^{\prime}(2)}{\zeta^{}(2)}\r) + U(x) 
+ O\l(x^{1/2}\delta(x)\r),    
\end{align*}
where 
$$
\delta(x)= {\rm exp}\l(-c\l(\log x \log\log x\r)^{1/3}\r) 
$$ with $c$ being a positive constant. We divide $U(x)$ into two sums as follows:
\begin{align}    
\label{A1}
U(x) &=\l(\sum_{n\leq x^{1/2}\delta(x)}+\sum_{x^{1/2}\delta(x)< n\leq x^{1/2}}\r)\mu*\mu(n) \Delta \l(\frac{x}{n^2}\r) \nonumber \\
& := U_{1}(x) + U_{2}(x),  \qquad {\rm say}.
\end{align}
Using the well-known result $\Delta(x)=O(x^{1/3})$ and partial summation, we obtain   
\begin{align}      
\label{A2}  
U_{1}(x)&\ll \sum_{n\leq x^{1/2}\delta(x)}\tau(n)\Delta\l(\frac{x}{n^2}\r) \nonumber \\
         &\ll x^{1/3}\sum_{n\leq x^{1/2}\delta(x)}\frac{\tau(n)}{n^{2/3}}  \nonumber \\
&\ll x^{1/3}\l( x^{1/2}\delta(x)\r)^{1/3} \log x^{1/2}\delta(x)    \nonumber  \\ 
&\ll x^{1/2}{\rm exp}\l(-c' (\log x \log\log x)^{1/3}\r)  
\end{align}
with $c'$ being a positive constant. Next, employing the formula  
$\int_{1}^{y}\Delta(\xi)d\xi =\frac{y}{4} + O(y^{3/4})$
for any $y>1$ (see \cite{HI}), we deduce   
\begin{align}  \label{A3}
U_{2}(x) 
& \ll  \max_{x^{1/2}\delta(x)< n\leq x^{1/2}}|\mu*\mu(n)| \sum_{x^{1/2}\delta(x)< n\leq x^{1/2}} \Delta \l(\frac{x}{n^2}\r) \nonumber \\
& \ \ \ll  x^{\varepsilon}\int_{1}^{1/\delta^{2}(x)}\Delta(\xi)d\xi  \nonumber \\
& \qquad \ll \frac{x^\varepsilon}{\delta^{2}(x)} \ll x^{1/3+\varepsilon}.
\end{align}
Substituting \eqref{A2} and \eqref{A3} into \eqref{A1}, we establish the assertion of Theorem \ref{th31}.
%%%%%%%%%%%%%%%%%%%%%%%%%%%%%%%%%%%%%%%%%%%%%%%%%%%%%%%%%%%%%%%%%%%%%%%%%%%%%%%%%%%%%%%%%%%%%%%%%%%%%%%%%%%%%%%%%%%%%%%%%%%%%%%%%

\subsection{Proof of Theorem~\ref{th4}} 

%%%%%%%%%%%%%%%%%%%%%%%%%%%%%%%%%%%%%%%%%%%%%%%%%%%%%%%%%%%%%%%%%%%%%%%%%%%%%%%%%%%%%%%%%%%%%%%%%%%%%%%%%%%%%%%%%%%%%%%%%%%%%%
From \eqref{mu-D} with $k=2$, we observe that 
\begin{align}                                                       \label{s-full}
\sum_{n=1}^{\infty}\frac{f_{\mu,2}(n)}{n^s} = \sum_{n=1}^{\infty}\l(\sum_{d^{2}m=n}\mu*\mu(d)\tau(m)\r)\frac{1}{n^s}=\frac{\zeta^{2}(s)}{\zeta^{2}(2s)},
\end{align} 
which converges absolutely for $\Re (s) > 1$. Without loss of generality, we can assume that $x\in \mathbb{Z}+\frac12$.
From \eqref{s-full}, we derive  
\begin{align*}                                                      
%\label{S-sum}
S_{\mu}(x)&:=\sum_{n\leq x}f_{\mu,2}(n)  
             =\sum_{n\leq x}\sum_{n=m^{}\ell^{2}}\mu*\mu(\ell)\tau(m). 
\end{align*}
Let $x^{\varepsilon} < Y < x^{\frac{1}{2}-\varepsilon}$. Then the above formula can be rewritten as follows   
\begin{align}                                                     
\label{Eq3-2} 
S_{\mu}(x)
&= \left(\sum_{m\ell^{2}\leq x \atop \ell\leq Y}+ \sum_{m\ell^{2}\leq x \atop \ell  > Y}\right)\mu*\mu(\ell)\tau(m)   \nonumber   \\
&:= S_{1}(x;Y) + S_{2}(x;Y)  \qquad  \quad \text{ say}.     
\end{align}
We define
$$
g_{Y}(s) :=\sum_{n> Y}\frac{\mu*\mu(n)}{n^s} = \frac{1}{\zeta^{2}(s)} -\sum_{n\leq Y}\frac{\mu*\mu(n)}{n^s} 
$$
for $\Re (s)>1$. Following a similar approach as in Lemma 3 in \cite{N}, we find that 
\begin{align}                                                       
\label{MV-g}
g_{Y}(s) = O\l(Y^{\frac12-\sigma+\varepsilon}(1+|t|)^{\varepsilon}\r)
\end{align}
is valid for $\sigma \geq \frac12+\varepsilon$ under the Riemann Hypothesis. 
Applying \eqref{s-full3}, we have 
\footnote{
Alternative expressions for the first and second sums on the right-hand side of  \eqref{Eq3-21} are      
$$
\underset{s=1}{\rm{Res}}~\frac{\zeta^{2}(s)}{\zeta^{2}(2s)}\frac{x^{s}}{s} -
\underset{s=1}{\rm{Res}}~\zeta^{2}(s)g_{Y}(2s)\frac{x^{s}}{s}.
$$}
\begin{align}                                                                             \label{Eq3-21}
&S_{1}(x;Y) =\sum_{\ell\leq Y}\mu*\mu(\ell)D\left(\frac{x}{\ell^2}\right)  \nonumber \\
&=x\sum_{\ell\leq Y}\frac{\mu*\mu(\ell)}{\ell^2}\log \frac{x}{\ell^2} + (2\gamma-1)x\sum_{\ell\leq Y}\frac{\mu*\mu(\ell)}{\ell^2} 
 + \sum_{\ell\leq Y}\mu*\mu(\ell)\Delta\left(\frac{x}{\ell^2}\right).           
\end{align}

To estimate $S_2(x;Y)$, we employ the methodology outlined in the proof of Theorem~\ref{th1}.
Let  $T$  be  a real parameter  at our disposal. 
Employing  Perron's formula (see Section \ref{Section2}) and \eqref{MV-g}, we deduce  
\begin{align*}                                                                          
S_{2}(x;Y)  
&= \frac{1}{2\pi i} \int_{\alpha-iT}^{\alpha+iT} 
\zeta^{2}(s) g_{Y}(2s) \frac{x^s}{s^{}}ds 
+ O\left(\frac{x^{1+\varepsilon}}{T}\right),  
\end{align*}
with $\alpha = 1 +\frac{1}{\log x}$.  
Now, let $\Gamma(\alpha, \beta,T)$ denote the contour consisting of the line segments 
         $[\alpha-iT, \beta-iT]$, 
         $[\beta-iT, \beta+iT]$ 
   and   $[\beta+iT, \alpha+iT]$. 
We shift the integration with respect to $s$ to  $\Gamma(\alpha,\beta,T)$   with 
$
\frac14<\beta<\frac12. 
$  
The integrals over the horizontal line segments are denoted by $S_{2,1}$ and $S_{2,3}$, and  the integral over the vertical line segment by $S_{2,2}$.                        
The main term arises from the sum of the residue at the double pole $s=1$.   
Hence, by the Cauchy residue theorem,   we obtain  
\begin{align}                                                                             
\label{S_11}              
  S_{2}(x;Y) 
&=\frac{1}{2\pi i}
\left\{\int_{\beta+iT}^{\alpha+iT}+ \int_{\beta-iT}^{\beta+iT} +  \int_{\alpha-iT}^{\beta-iT}\right\}
\zeta^{2}(s)g_{Y}(2s) \frac{x^{s}}{s}ds         \nonumber    \\
&+ \underset{s=1}{\rm{Res}}~\zeta^{2}(s)g_{Y}(2s)  \frac{x^{s}}{s}    + O\l(\frac{x^{1+\varepsilon}}{T}\r).     
\end{align}
By applying Lemmas \ref{lem2} and \ref{lem3}, we deduce that  
\begin{align}                                                                             
\label{Q1}                    
&   S_{2,1},\  S_{2,3}  \ll  \frac{1}{T} \l(\int_{\beta}^{\frac12}+\int_{\frac12}^{\alpha} \r)
|\zeta(\sigma+iT)|^{2}~|g_{Y}(2(\sigma+iT))| x^{\sigma}d\sigma   \nonumber  \\
&\ll  Y^{1/2+\varepsilon}T^{2\varepsilon} 
      \int_{\beta}^{\frac12} \l(\frac{x}{T^{2}Y^2}\r)^{\sigma} d\sigma  
+     Y^{1/2+\varepsilon}T^{-1+2\varepsilon} 
      \int_{\frac12}^{\alpha} \l(\frac{x}{Y^2}\r)^{\sigma} d\sigma  \nonumber   \\
&\ll x^{}T^{-1+2\varepsilon}Y^{-3/2+\varepsilon} 
   + x^{1/2}T^{-1+2\varepsilon}Y^{-1/2+\varepsilon} 
   + x^{\beta}T^{-2\beta+2\varepsilon}Y^{1/2-2\beta+\varepsilon}.  
\end{align}
Consider the integral $S_{2,2}$. Let $T\geq T_{0}$, where $T_0$ is a sufficiently large positive number. Using \eqref{MV-g}, and Lemmas \ref{lem2} and \ref{lem3}, we deduce  
\begin{align}                                                                             \label{ki1}
& S_{2,2} \ll x^{\beta}Y^{1/2-2\beta+\varepsilon}  + x^{\beta} \int_{T_{0}\leq |t| \leq T} 
\frac{|\chi(\beta+it)|^{2}|\zeta(1-\beta-it)|^{2}|g_{Y}(2\beta+2it)|}{1+|t|} dt  \nonumber  \\
&\qquad \ll   x^{\beta}T^{1-2\beta+2\varepsilon}Y^{\frac12-2\beta+\varepsilon}.    
\end{align}
Therefore, by substituting  \eqref{Eq3-21}--\eqref{Q1} and \eqref{ki1} into \eqref{Eq3-2} and setting $T=x^{\frac12}Y^{\frac{4\beta-1}{4-4\beta}-\varepsilon}$, we obtain
\begin{align*}                                               
%\label{K1}
S_{\mu}(x) = \frac{x}{\zeta^{2}(2)}\left(\log x  +2\gamma -1 -4 \frac{\zeta^{\prime}(2)}{\zeta^{}(2)}\right)   
+ \sum_{\ell\leq Y}\mu*\mu(\ell)\Delta\left(\frac{x}{\ell^2}\right) 
+ O\left(x^{\frac12+\varepsilon}Y^{\frac{1-4\beta}{4-4\beta}+\varepsilon} \right),  
\end{align*}
where $x^{\varepsilon} < Y < x^{\frac{1}{2}-\varepsilon}$. This completes the proof.
%%%%%%%%%%%%%%%%%%%%%%%%%%%%%%%%%%%%%%%%%%%%%%%%%%%%%%%%%%%%%%%%%%%%%%%%%%%%%%%%%%%%%%%%%%%%%%%%%%%%%%%%%%%%%%%%

\subsection{Proof of Corollary \ref{cor1}}

%%%%%%%%%%%%%%%%%%%%%%%%%%%%%%%%%%%%%%%%%%%%%%%%%%%%%%%%%%%%%%%%%%%%%%%%%%%%%%%%%%%%%%%%%%%%%%%%%%%%%%%%%%%%%%%%%%%
Given that  
$$
\frac{1-4\beta}{4-4\beta}=-\frac14, \quad  \left(\text{taking $\beta=\frac25$}\right), 
$$ 
the error term in Theorem \ref{th4} can be estimated as 
\begin{align}                                    
\label{M1}
\ll x^{1/2+\varepsilon}Y^{-1/4+\varepsilon}.
\end{align}
This allows us to directly estimate the sum in Theorem \ref{th4}, which is
\begin{align}                                    
\label{M2}
\sum_{\ell\leq Y}\mu*\mu(\ell)\Delta\l(\frac{x}{\ell^2}\r) \ll x^{\theta+\varepsilon}Y^{1-2\theta+\varepsilon},
\end{align}
by applying \eqref{Delta}. From \eqref{M1} and \eqref{M2}, setting
$
Y=x^{\frac{2-4\theta}{5-8\theta}}, 
$  
with $\theta=\frac{131}{416}$
we obtain
$$
E_{\mu}(x) \ll x^{\frac{439}{1032}+\varepsilon}.
$$
This completes the proof.
%%%%%%%%%%%%%%%%%%%%%%%%%%%%%%%%%%%%%%%%%%%%%%%%%%%%%%%%%%%%%%%%%%%%%%%%%%%%%%%%%%%%%%%%%%%%%%%%%%%%%%%%%%%%%%%%%%%%%%%%%%%%%%%%%%%%%%%

\subsection{Proof of Theorem \ref{th5}}   

%%%%%%%%%%%%%%%%%%%%%%%%%%%%%%%%%%%%%%%%%%%%%%%%%%%%%%%%%%%%%%%%%%%%%%%%%%%%%%%%%%%%%%%%%%%%%%%%%%%%%%%%%%%%%%%%%%%%%%%%%%%%%%%%%%%%%%%%%
We aim to explore the mean square estimate for $E_{\mu}(x)$ under the Riemann Hypothesis.
Our analysis begins with \eqref{mu-D}, where setting $k=2$
yields the following relation
$$\sum_{n\geq 1}\frac{f_{\mu}(n)}{n^s}=\frac{\zeta^2(s)}{\zeta^2(2s)}.$$
Employing the Perron formula, as detailed in Section \ref{Section2}, we derive that   
\begin{align}                                                    
\label{g-21}
S_{\mu}(x) 
&= \frac{1}{2\pi i}\int_{2-i\infty}^{2+i\infty}\frac{\zeta^{2}(s)}{\zeta^{2}(2s)}\frac{x^{s}}{s} ds. 
\end{align} 
Upon shifting the line of integration in \eqref{g-21} to some $1/2<\alpha<1$, we encounter a pole at $s=1$ within the integrand. Using the Cauchy residue theorem, we get 
$$                                                               
S_{\mu}(x)  
= \underset{s=1}{\rm{Res}}~\frac{\zeta^{2}(s)}{\zeta^{2}(2s)} \frac{x^{s}}{s} + \frac{1}{2\pi i}\int_{\alpha-i\infty}^{\alpha+i\infty}\frac{\zeta^{2}(s)}{\zeta^{2}(2s)s}\l(\frac{1}{x}\r)^{-s}ds, 
$$  
As previously discussed, the residue at $s=1$ is given by $$\underset{s=1}{\rm{Res}}~\frac{\zeta^{2}(s)}{\zeta^{2}(2s)} \frac{x^{s}}{s}=
\frac{x}{\zeta^{2}(2)}\left(\log x  +2\gamma -1 -4 \frac{\zeta^{\prime}(2)}{\zeta^{}(2)}\right).
$$
We have
\begin{align}                                                    
\label{g-Delta21}
E_{\mu}(x) & = 
\frac{1}{2\pi i}\lim_{T\to \infty}\int_{\alpha-iT}^{\alpha+iT}\frac{\zeta^{2}(s)}{\zeta^{2}(2s)s}\l(\frac{1}{x}\r)^{-s}ds.
\end{align}
We define $\alpha_{0}$ as the infimum of $\alpha >0$ for which  
\begin{equation}
\label{eq12}
\int_{-\infty}^{\infty}\frac{|\zeta^{}(\alpha+it)|^{4}}{|\zeta(2\alpha+2it)|^4|\alpha+it|^2}dt \ll 1.
\end{equation}
Our goal is to show that for $\alpha>\alpha_0$ and approaching to $1$, the following holds:
\begin{equation}
\label{eq11}
  \frac{1}{2\pi}\int_{-\infty}^{\infty}
\frac{|\zeta^{}(\alpha+it)|^{4}}{|\zeta(2\alpha+2it)|^4 |\alpha+it|^2}dt  
 =\int_{0}^{\infty}E_{\mu}^{2}(x) x^{-2\alpha-1}dx.  
\end{equation}                                                                
From \eqref{g-Delta21}, since ${\zeta^{2}(s)\zeta^{-2}(2s)}{s^{-1}} \to  0$ uniformly in the strip
as $t\to \pm\infty$, and by integrating over the rectangular path passing through the vertices, $\alpha'- iT, \alpha- iT, \alpha +iT$ and $\alpha'+ iT$ where $\alpha_0<\alpha' <\alpha < 1$, we confirm that \eqref{g-Delta21} is valid for any $\alpha>\alpha_0$. By substituting $x$ with $1/x$, and applying Parseval's identity, we achieve the desired equality in \eqref{eq11}. Consequently, if \eqref{eq12} is satisfied for some $\alpha$\ ($\alpha_0<\alpha<1$), then for any $T\geq 5$ we have 
$$                                                                            
\int_{T/2}^{T}E_{\mu}^{2}(x)dx \ll T^{2\alpha+1}.
$$ 
To finalize the proof of our theorem, we must demonstrate that \eqref{eq12} is attainable. Assuming $\alpha_{0}<\alpha <1/2$ as a parameter to be precisely chosen, and for any $T\geq 5$, we employ  Lemma \ref{lem2} to obtain    
\begin{align*}
L(T)&:=\int_{T/2}^{T}\frac{|\zeta^{}(\alpha+it)|^{4}}{|\zeta(2\alpha+2it)|^4 |\alpha+it|^2}dt  \nonumber \\
&=\int_{T/2}^{T}
\frac{|\chi(\alpha+it)|^{4}|\zeta(1-\alpha+it)|^{4}}{ |\zeta(2\alpha+2it)|^4 |\alpha+it|^2}dt \nonumber \\
&\ll \int_{T/2}^{T}
\frac{t^{2-4\alpha}|\zeta(1-\alpha+it)|^{4}}{ |\zeta(2\alpha+2it)|^4 |\alpha+it|^2}dt. \nonumber 
\end{align*}
By setting $\alpha = 1/4+3\varepsilon$, for any sufficiently small $\varepsilon>0$ (choosing $\varepsilon:=\frac{c}{\log\log T}$ with a positive constant $c$), and applying Lemma \ref{lem3}, we deduce
\begin{align*}
L(T) \ll &\int_{T/2}^{T}\frac{|\zeta^{}(\frac34-3\varepsilon-it)|^{4}}{|\zeta(\frac12+6\varepsilon+2it)|^4 t^{1+12\varepsilon}}dt 
     \ll \int_{T/2}^{T}\frac{1}{t^{1+4\varepsilon}}dt  \ll 1.   
\end{align*}
Therefore, we conclude that
\begin{align*} 
\int_{1}^{T}E_{\mu}^{2}(u)du 
&=\sum_{k\leq \frac{\log T}{\log 2}}\int_{2^{k-1}}^{2^k}E_{\mu}^{2}(u)du  
\ll T^{3/2+6\varepsilon} \\
&\ll T^{3/2}{\rm exp}\l(A\frac{\log T}{\log\log T}\r),
\end{align*} 
with a positive real number $A$.
%%%%%%%%%%%%%%%%%%%%%%%%%%%%%%%%%%%%%%%%%%%%%%%%%%%%%%%%%%%%%%%%%%%%%%%%%%%%%%%
%%%%%%%%%%%%%%%%%%%%%%%%%%%%%%%%%%%%%%%%%%%%%%%%%%%%%%%%%%%%%%%%%%%%%%%%%%%%%%%%%%%%%%%%%%%%%%%%%%%%%%%%%%%%%%%%%%%%%%%%%%%%%%%%%%%%%%%%%%%%%%%%%

\section*{Acknowledgement}
The first author is supported by JSPS Grant-in-Aid for Scientific Research (C)(21K03205).
The second author is supported by the Austrian Science Fund (FWF): Project N. 35863.

%%%%%%%%%%%%%%%%%%%%%%%%%%%%%%%%%%%%%%%%%%%%%%%%%%%%%%%%%%%%%%%%%%%%%%%%%%%%%%%%%%%%%%%%%%%%%%%%%%%%%%%%%%%%%%%%%%%%%%%%%%%%%%%%%%%%%%%%%%%%%%%%%

\medskip\noindent {\footnotesize Isao Kiuchi: Department of Mathematical Sciences, Faculty of Science,
Yamaguchi University, Yoshida 1677-1, Yamaguchi 753-8512, Japan. \\
e-mail: {\tt kiuchi@yamaguchi-u.ac.jp}}

\medskip\noindent {\footnotesize Sumaia Saad Eddin: 
Johann Radon Institute for Computational and Applied Mathematics, Austrian Academy of Sciences, Altenbergerstrasse 69, 4040 Linz, Austria.\\
e-mail: {\tt sumaia.saad-eddin@ricam.oeaw.ac.at}}

\end{document}